\documentclass[oneside,12pt]{amsart}

\usepackage{amssymb, mathrsfs, amscd}
\usepackage[all]{xy}
\usepackage{graphics}
\usepackage{graphicx}
\usepackage{enumerate}

\newtheorem{theorem}{Theorem}[section]

\newtheorem{lemma}[theorem]{Lemma}

\newtheorem{corollary}[theorem]{Corollary}

\newtheorem{remark}[theorem]{Remark}

\newtheorem*{question}{Question}
\newtheorem{cor}{Corollary}
\newtheorem*{thma}{Theorem A}
\newtheorem*{thmb}{Theorem B}
\newtheorem*{thmc}{Theorem C}

\author{Dieter Degrijse}
\title[]{Amenable groups of finite cohomological dimension and the zero divisor conjecture}
\address{School of Mathematics, Statistics and Applied Mathematics, NUI Galway, Ireland}
\email{dieter.degrijse@nuigalway.ie}

\begin{document}

\begin{abstract} We prove that every amenable group of cohomological dimension two whose integral group ring is a domain is solvable and investigate certain homological finiteness properties of groups that satisfy the analytic zero 
divisor conjecture and act on an acyclic $CW$-complex with amenable stabilisers. \end{abstract}

\maketitle

\section{Introduction}
All groups appearing in this paper are assumed to be discrete. Recall that the class of elementary amenable groups is by definition the smallest class of groups that contains all finite and all abelian groups and that is closed under 
directed unions, extensions, quotients and subgroups. A group $G$ is said to be amenable if for every continuous action of $G$ on a compact, Hausdorff space $X$, there is a $G$-invariant probability measure on $X$. As the name 
suggest, all elementary amenable groups are amenable (e.g.~see \cite[Appendix G]{Bekka}) and by now there is a multitude of examples of groups that are amenable but not elementary amenable. Such groups arise for example in 
the class of branch groups, which includes the first example by Grigorchuk of a group of intermediate growth (see \cite{BGS},\cite{jus}). Another important source of amenable non-elementary amenable groups are (commutator 
subgroups of) topological full groups of Cantor minimal systems (see \cite{jusmon}) and we refer to the introduction of  \cite{jus} for more examples.

 However, as far as we are aware none of the currently known examples of amenable but not elementary amenable groups are known to have finite cohomological dimension over any field. For example, many Branch groups, including 
Grigorchuk's example, are commensurable to an $n$-fold direct product of themselves for some $n\geq 2$ (see \cite{BGS}). By Corollary \ref{cor: comm}, the cohomological dimension of such a non locally-finite group is infinite over any 
field. Since commutator subgroups of topological full groups of Cantor minimal systems contain $\bigoplus_{n\geq 0}\mathbb{Z}$ as a subgroup (see \cite[Remark 4.3(2)]{GriMed}), these groups also have infinite cohomological 
dimension over any field. Hence, one might well wonder if in fact all amenable groups of finite cohomological dimension (over some ring or field) are elementary amenable. 

We will be working with cohomological dimension over the integers, which we will therefore just refer to as cohomological dimension. Since all elementary amenable groups of finite cohomological dimension must be torsion-free and 
have finite Hirsch length by \cite[Lemma 2]{Hillman2}, one can conclude from \cite[Corollary 1]{HillmannLinnel} that they are virtually solvable. We are therefore led to the following question, which was brought to the authors attention 
by Peter Kropholler. 

\begin{question} Is every amenable group of finite cohomological dimension virtually solvable?
\end{question}
By Stalling's theorem (\cite{Stallings}) this question obviously has a positive answer for groups of cohomological dimension at most one.  By the Tits alternative, all amenable linear groups of finite cohomological dimension are 
virtually solvable. Also note that the answer to the more general question where one replaces `amenable' by `does not contain non-abelian free subgroups' is no. Indeed, Olshanskii's example of a torsion-free Tarski monster, i.e.~ a 
finitely generated torsion-free non-cyclic group all of whose proper subgroups are cyclic, is non-amenable and admits an aspherical presentation, showing that it has cohomological dimension two (e.g.~see \cite{Ol}). \\

Our main results, which we will state below, depend on the validity of Kaplansky's zero divisor conjecture for group rings or a generalisation of it due to Linnell (e.g.~see \cite{Linnell}), called the analytic zero divisor conjecture, which states 
that if $
\alpha \in \mathbb{C}[G]\setminus \{0\}$ and $\beta \in l^2(G)\setminus\{0\} $, then $\alpha\beta\neq 0$. There are no known counterexamples to the (analytic) zero divisor conjecture and it has been proven for a wide class of groups, 
including torsion-free elementary amenable groups (see \cite[Th. 1.4]{KLM}) and extensions of right orderable groups by torsion-free elementary amenable groups (e.g.~see \cite[Th. 8.2]{Linnell}). Moreover, Elek showed in \cite{Elek} 
that for finitely generated amenable groups the Kaplansky zero divisor conjecture over $\mathbb{C}$ and the analytic zero divisor conjecture are equivalent. We also remark that for a torsion-free group $G$, $\mathbb{Z}[G]$ is a 
domain if and only if $\mathbb{Q}[G]$ is a domain and for any field $\mathbb{Z}\subseteq k \subseteq \mathbb{C}$, the Atiyah conjecture of order $\mathbb{Z}$ implies that $k[G]$ is a domain (e.g. see \cite[Lemma 10.15]
{LuckL2Book}). \\

Our first main result concerns the homological finiteness properties $FP_n$ and $FP$ of a group. These notions will be recalled in the next section.

\begin{thma} Let $G$ be an amenable group such that $\mathbb{Z}[G]$ is a domain. If $G$ has cohomological dimension $n$ and is of type $FP_{n-1}$, then $G$ is of type $FP$.

\end{thma}
Note that the conclusion of Theorem A is certainly false in general for non-amenable torsion-free groups. Indeed, for  $n\geq 1$, the Bestvina-Brady group $H_L$ associated to any flag-triangulation $L$ of the $(n-1)$-sphere contains a 
non-abelian free subgroup, has cohomological dimension $n$ by \cite[Th. 22]{LearySaad} and is of type $FP_{n-1}$ but not of type $FP$ by the main theorem of \cite{BestvinaBrady}. Also, the Tarski monster mentioned above is 
finitely generated and of cohomological dimension two, but is not finitely presented. \\

Using Theorem A we obtain a positive answer to the question above in the $2$-dimensional case, assuming the validity of the zero-divisor conjecture. This generalises \cite[Th. 3]{KLL}.

\begin{thmb} Every amenable group $G$ of cohomological dimension $2$ such that $\mathbb{Z}[G]$ is a domain is solvable and hence isomorphic to a solvable Baumslag-Solitar group $BS(1,m)$ for some non-zero $m \in \mathbb{Z}
$ or to a non-cyclic subgroup of the additive rationals.

\end{thmb}

A modified version of the proof of Theorem A gives the following theorem.

\begin{thmc} Let $G$ be a group of cohomological dimension $n$ that satisfies the analytic zero-divisor conjecture and admits an acyclic $G$-CW-complex $X$ with amenable stabilisers. If $G$ is of type $FP_{n-1}$ and $X$ does not 
have a free $G$-orbit of $n$-cells, then $G$ is of type  $FP$.
\end{thmc}
Since an $n$-dimensional $G$-CW-complex obviously has no free $G$-orbit of $(n+1)$-cells, the following corollary is immediate.
\begin{cor} Let  $\Gamma$ be a group that admits an acyclic $n$-dimensional $\Gamma$-CW-complex with amenable stabilisers. Then any subgroup $G$ of $\Gamma$ of cohomological dimension $n+1$ that satisfies the analytic 
zero-divisor conjecture and is of type $FP_{n}$ is of type $FP$.
\end{cor}
Recall that a group is said to be almost coherent, if every finitely generated subgroup is almost finitely presented, i.e.~of type $FP_2$.
\begin{cor} Fundamental groups of graphs of groups with cyclic edge groups and vertex groups that are either subgroups of the additive rational numbers or solvable Baumslag-Solitar groups are almost coherent if they satisfy the 
analytic zero divisor conjecture.

\end{cor}

A group $G$ is said to commensurate a subgroup $H$ if $H^{g}\cap H$ has finite index in both $H$ and $H^{g}$ for every element $g \in G$. In particular, if $H$ is normal in $G$ then $G$ commensurates $H$. Using Theorem $C$, 
we prove the following.

\begin{cor} Let $G$ be a group that satisfies the analytic zero-divisor conjecture and commensurates a non-trivial amenable subgroup. If $G$ has cohomological dimension $n$ and is of type $FP_{n-1}$, then $G$ is of type $FP$. 

\end{cor}

\noindent \textbf{Acknowledgements.} The author is grateful to Peter Kropholler for introducing him to the question posed in this paper and for valuable discussions and remarks concerning a preliminary version of this work. This 
research was partially supported by the Danish National Research Foundation through the Centre for Symmetry and Deformation (DNRF92).

\section{Finiteness conditions}

We start by recalling some basic facts about homological finiteness properties of groups and refer the reader to \cite{BieriBook}, \cite{Brown} and \cite{Weibel} for more details. Let $G$ be a group, let $R$ be a commutative ring with unit 
and let $R[G]$ denote the group ring of $G$ with coefficients in $R$. In this paper we will always consider left $R[G]$-modules, unless stated otherwise. If there is no mention of a ring $R$, then $R$ is assumed to equal $\mathbb{Z}$.

The cohomological dimension $\mathrm{cd}_R(G)$ of $G$ over $R$ is by definition the length of the shortest projective $R[G]$-resolution of the trivial $R[G]$-module $R$. If there does not exist such a finite length projective 
resolution, then the cohomological dimension is by definition infinite. Equivalently, $\mathrm{cd}_R(G)$ is the largest integer $n\geq 0$ for which there exists an $R[G]$-module $M$ such that the $n$-th cohomology group $\mathrm{H}
^n(G,M)$ is non-zero. The homological dimension $\mathrm{hd}_R(G)$ of $G$ over $R$ is by definition the length of the shortest flat $R[G]$-resolution over the trivial $R[G]$-module $R$. Again, if there does not exist a finite length 
flat resolution then the homological dimension is by definition infinite. Equivalently, $\mathrm{hd}_R(G)$ is the largest integer $n\geq 0$ for which there exists a right $R[G]$-module $M$ such that the $n$-th homology group $
\mathrm{H}_n(G,M)$ is non-zero. In general one has $\mathrm{hd}_R(G)\leq \mathrm{cd}_R(G)$, and if $G$ is countable one also has $\mathrm{cd}_R(G)\leq \mathrm{hd}_R(G)+1$ (e.g.~see \cite[Th 4.6]{BieriBook}). Note that $
\mathrm{cd}_{R}(G)\leq \mathrm{cd}(G)$, and that $\mathrm{cd}(G)<\infty$ implies that $G$ is torsion-free.

The group $G$ is said to be of type $FP_n$ over $R$ for an integer $n\geq 0$ if and only if there exists a projective $R[G]$-resolution $P_{\ast}$ of the trivial $R[G]$-module $R$ such that $P_d$ is a finitely generated $R[G]$-module 
for all $ 0\leq d \leq n$. The group $G$ is said to be of type $FP$ over $R$ if and only if there exists a finite length projective $R[G]$-resolution $P_{\ast}$ of the trivial $R[G]$-module $R$ such that $P_d$ is a finitely generated $R[G]$-
module for all $d\geq 0$. Every group $G$ is of type $FP_0$ over $R$ and $G$ is of type $FP_1$ over $R$ if and only if $G$ is finitely generated.  If $G$ is finitely presented then $G$ is of type $FP_2$ over $R$ but not necessarily 
the other way around (see \cite[Prop. 2.1 and 2.2]{BieriBook} and \cite{BestvinaBrady}). Note also that if $G$ is of type $FP_n$ over $R$ and $\mathrm{cd}_R(G)=n$, then $G$ is of type $FP$ over $R$ (e.g.~see \cite[VIII Prop. 6.1]
{Brown}). \\

The following lemma is a variation on a classical criterion due to Strebel implying that a group is of type $FP_n$. It will be well-known to experts.

\begin{lemma} \label{lemma: direct limit} Let $R$ be a commutative ring with unit and let $G$ be a group with $\mathrm{cd}_R(G)=n$ that is type $FP_{n-1}$ over $R$. Then $G$ is of type $FP$ over $R$ if and only if for any directed 
system of free $R[G]$-modules $\{F_i\}_{i \in I}$ with $\varinjlim{F_i}= 0$ one has 
\[      \varinjlim{\mathrm{H}^n(G,F_i)}=0.     \] 

\end{lemma}

\begin{proof} A classical result of Strebel (e.g.~see \cite[Th. 1.3.]{BieriBook}) implies that if $G$ is of type $FP_{n-1}$ over $R$ and $\mathrm{cd}_R(G)=n$ then  $G$ is of type $FP$ if and only if $\varinjlim{\mathrm{H}^n(G,M_i)}=0$ 
for any directed system of $R[G]$-modules $\{M_i\}_{i \in I}$ with $\varinjlim{M_i}= 0$. Let $\{M_i\}_{i \in I}$ be such a directed system. Now for each $i \in I$, consider the free $R[G]$-module 
\[   F_i = \bigoplus_{m \in M_i \setminus \{0\}} R[G]  \]
and the surjection of $R[G]$-modules
\[    \pi_i : F_i \rightarrow  M_i : e_m \mapsto m     \]
where $e_m$ denotes the identity element in the $R[G]$-factor corresponding to $m\in M$. Given $i\leq j$ in $I$ we have the corresponding $R[G]$-homomorphism $\varphi_{i,j}: M_{i}\rightarrow M_j$. Define the $R[G]$-
homomorphism 
\[      \theta_{i,j}: F_i \rightarrow F_j   :  e_m \mapsto \Big\{\begin{array}{cc}
 e_{\varphi_{i,j}(m)} & \mbox{if $\varphi_{i,j}(m)\neq 0 $} \\
0 & \mbox{otherwise.}
\end{array}  \]
One checks that the maps $\theta_{i,j}$ assemble to form a directed system $\{F_i\}_{i\in I}$ with vanishing direct limit and that 
$$\{\pi_i\}_{i \in I}: \{F_i\}_{i \in I} \rightarrow \{M_i\}_{i \in I} $$
is a surjective map of directed systems, i.e.~$\varphi_{i,j}\circ \pi_i = \pi_j \circ \theta_{i,j}$ for all $i\leq j$ in $I$. Since $\mathrm{cd}_R(G)=n$ we obtain from a collection of natural long exact cohomology sequences, a surjective map 
of directed systems
\[     \{\mathrm{H}^n(G,F_i)\}_{i \in I} \rightarrow   \{\mathrm{H}^n(G,M_i)\}_{i \in I} \rightarrow 0.  \]
Since taking direct limits is exact, there is a surjection
\[     \varinjlim{\mathrm{H}^n(G,F_i)} \rightarrow   \varinjlim{\mathrm{H}^n(G,M_i)} \rightarrow 0,      \]
from which the lemma easily follows.

\end{proof}
Recall that two groups are said to be abstractly commensurable if they have isomorphic subgroups of finite index. In the introduction we mentioned that non locally-finite groups $G$ that are abstractly commensurable to a direct product 
$G^n$ for some $n\geq 2$, have infinite cohomological dimension over any field $k$. The following lemma will allow us to prove this in the corollary below.
\begin{lemma} Let $k$ be a field and let $G_i$ be a group with $\mathrm{hd}_k(G_i)=n_i$ for $i=1,2$, then $\mathrm{hd}_k(G_1\times G_2)=n_1+n_2$. 

\end{lemma}
\begin{proof} Let $M_i$ be a right $k[G_i]$-module such that $\mathrm{H}_{n_i}(G_i,M_i)\neq 0$ for $i=1,2$ and consider right the $k[G_1\times G_2]$-module $M_1\otimes_k M_2$. By the Lyndon-Hochschild-Serre spectral sequence 
associated to the extension 
\[     1 \rightarrow G_1 \rightarrow G_1 \times G_2 \rightarrow G_2 \rightarrow 1   \]
it suffices to show that 
\[       \mathrm{H}_{n_2}(G_2,\mathrm{H}_{n_1}(G_1,M_1\otimes_k M_2)) \neq 0.   \]
Let $P_{\ast}$ be  a projective $k[G_1]$-resolution of the trivial $k[G_1]$-module $k$. Since
\[    (M_1 \otimes_k M_2)\otimes_{k[G_1]} P_{\ast}   \cong  (M_1 \otimes_{k[G_1]} P_{\ast} ) \otimes_k M_2 \] and $-\otimes_k M_2$ is an exact functor because $k$ is a field, we have
\[        \mathrm{H}_{n_1}(G_1,M_1\otimes_k M_2)\cong  \mathrm{H}_{n_1}(G_1,M_1 )\otimes_k M_2.   \]
A similar reasoning shows that
\[   \mathrm{H}_{n_2}(G_2, \mathrm{H}_{n_1}(G_1,M_1 )\otimes_k M_2)\cong \mathrm{H}_{n_1}(G_1,M_1)\otimes_k \mathrm{H}_{n_2}(G_2,M_2).      \]
We conclude that 
\[  \mathrm{H}_{n_2}(G_2,\mathrm{H}_{n_1}(G_1,M_1\otimes_k M_2))  \cong \mathrm{H}_{n_1}(G_1,M_1)\otimes_k \mathrm{H}_{n_2}(G_2,M_2) \neq 0,\]
proving the lemma.

\end{proof}

\begin{corollary}\label{cor: comm} Let $G$ be a group that is not locally finite. If $G$ is abstractly commensurable to a direct product $G^n$ for some $n\geq 2$, then $\mathrm{hd}_k(G)$, and hence also $\mathrm{cd}_k(G)$, is infinite 
for any field $k$. 
\end{corollary}
\begin{proof} Suppose $\mathrm{hd}_k(G)<\infty$. Then $\mathrm{hd}_k(G^n)< \infty$ and we have 
$\mathrm{hd}_k(G)=\mathrm{hd}_k(K)$ and $\mathrm{hd}_k(G^n)=\mathrm{hd}_k(L)$
for every finite index subgroup $K$ of $G$ and finite index subgroup $L$ of $G^n$ by \cite[Cor. 5.10]{BieriBook}. Since $G$ is abstractly commensurable with $G^n$, this implies that 
$\mathrm{hd}_k(G)=\mathrm{hd}_k(G^n)$. But then the lemma above says that
$n(\mathrm{hd}_k(G))=\mathrm{hd}_k(G^n)=\mathrm{hd}_k(G)$ for some $n\geq 2$, which implies that $\mathrm{hd}_k(G)=0$. We now conclude from \cite[Prop. 4.12(b)]{BieriBook} that $G$ is locally finite, which is a contradiction. 
This proves that $\mathrm{hd}_k(G)= \infty$.

\end{proof}

\section{Non-commutative localisation}
Later on we will need to embed the group ring $\mathbb{Z}[G]$ of a certain group $G$, into a ring $D$ such that every non-zero element of $\mathbb{Z}[G]$ is invertible in $D$.  In this section we recall how to construct such a ring $D$, assuming $G$ satisfies the assumptions of either Theorem A or Theorem B. \\

We start by briefly reviewing the Ore condition and the process of Ore localisation and refer the reader to \cite[4.10]{Lam} for details and proofs. Let $R$ be a ring with a unit and let $S$ be a multiplicatively closed subset of $R$ that 
contains the unit of $R$ but does not contain any zero divisors of $R$. The set $S$ is said to satisfy the left Ore condition with respect to $R$ if for every $s \in S$ and every $a \in R$, there exists $t \in S$ and $b \in R$ such that 
$ta=bs$. If $S$ satisfies the left Ore condition with respect to $R$ one can construct the ring $$S^{-1}R = S\times R / \sim$$
where $(s,a)\sim (s',a')$ if there exist $b,b' \in R$ such that $bs=b's' \in S$ and $ba=b'a' \in R$. One denotes an equivalence class containing $(s,a)$ suggestively by $\frac{a}{s}$ and defines
\[   \frac{a_1}{s_1}+\frac{a_2}{s_2}=\frac{sa_1+ra_2}{ss_1}     \]
where $r\in R$ and $s\in S$ are such that $r s_2=ss_1\in S$, and
\[     \frac{a_1}{s_1}\cdot \frac{a_2}{s_2}= \frac{ra_2}{ss_1} \]
where $r\in R$ and $s\in S$ are such that $r s_2=sa_1\in R$. One can check that this turn $S^{-1}R$ into a ring equipped with an injective ring homomorphism
\[  \varphi: R \rightarrow S^{-1}R : r \mapsto (1,r)  \]
such that the image under $\varphi$ of every element in $S$ is invertible in $S^{-1}R$. We will identify $R$ with its image $\varphi(R)$ in $S^{-1}R$. The ring $S^{-1}R$ is called the left Ore localisation of $R$ at $S$. Finally, note that 
$S^{-1}R$ is a flat right $R$-module and that we can also localise any left $R$-module $M$ at $S$ by defining
\[     S^{-1}M = S^{-1}R \otimes_R M.    \]
If $R$ is a domain and the set of all non-zero elements of $R$ satisfies the left Ore condition with respect to $R$, then $R$ is called a left Ore domain. \\

Returning to our specific setting, note that if $G$ is a torsion-free group such that $k[G]$ is a domain for a field $k$, then $k[G]$ is a left Ore domain if $G$ is amenable by \cite{Tamari}. In fact, this is an if and only if by the appendix of 
\cite{B} due to Kielak. When exploring these references, note that one can also consider the `right' Ore condition and all corresponding `right' notions but that for group rings (or more generally rings with an involution), the left Ore 
condition and the right Ore condition are equivalent and the corresponding localisations are isomorphic. For a group $G$ one easily verifies that $\mathbb{Z}[G]$ is a left Ore domain if and only if $\mathbb{Q}[G]$ is a left Ore domain. We may therefore  conclude that if $G$ is an amenable group such that $\mathbb{Z}[G]$ is a domain, then $\mathbb{Z}[G]$ is a left Ore domain and can therefore be embedded in a ring $D=S^{-1}\mathbb{Z}[G]$ ($S=\mathbb{Z}[G]\setminus \{0\}$) such that every non-zero element of $\mathbb{Z}[G]$ is invertible in $D$. This deals with the setup of Theorem $A$. In Theorem $C$ however, $G$ does not need to be amenable. Next, we explain how to proceed in that situation. 

Let $G$ be a torsion-free group that satisfies the analytic zero divisor conjecture and consider the Hilbert spaces

\[  l^2(G)= \{ \sum_{g \in G} a_g g \ (a_g \in \mathbb{C}) \ | \ \sum_{g \in G} |a_g|^2 <\infty\} \]
and 
\[  l^{\infty}(G)= \{ \sum_{g \in G} a_g g \ (a_g \in \mathbb{C}) \ | \ \sup_{g \in G} |a_g| <\infty\}. \]
One has the obvious inclusion 
$ \mathbb{C}[G] \subseteq l^2(G) \subseteq l^{\infty}(G)           $
and the natural ring multiplication on $\mathbb{C}[G]$ extends to a map
\[    l^{2}(G)\times l^{2}(G) \rightarrow l^{\infty}(G): (\alpha=\sum_{g \in G} a_g g,\beta=\sum_{h \in G} b_h h) \mapsto \alpha\beta = \sum_{g\in G} \Big( \sum_{h\in H}a_{gh^{-1}}b_h\Big)g.   \]
The group von Neumann algebra $\mathcal{N}(G)$ can be defined as (e.g. see \cite[Section 8]{Linnell}) as
\[   \mathcal{N}(G)= \{ \alpha \in l^{2}(G) \ | \ \alpha\beta \in l^{2}(G), \ \forall \beta \in l^{2}(G)\}     \]
and we have an inclusion of rings 
$\mathbb{C}[G] \subseteq \mathcal{N}(G)$
such that no non-zero element of $\mathbb{C}[G]$ is a zero-divisor in $\mathcal{N}(G)$, by the analytic zero divisor conjecture. Since the set $S$ of non-zero divisors in $\mathcal{N}(G)$ satisfies the left (and right) Ore condition in $
\mathcal{N}(G)$ (e.g. see \cite[Th. 8.22]{LuckL2Book}), one can consider the left (or right) Ore localisation $S^{-1}\mathcal{N}(G)$ which turns out to be isomorphic to the algebra $\mathcal{U}(G)$ of operators affiliated to $\mathcal{N}
(G)$ (e.g.~see \cite[Ch. 8]{LuckL2Book} ). Important for our purposes is that, under the assumptions of Theorem C, we have found a ring $\mathcal{U}(G)$ containing $\mathbb{C}[G]$ (and hence $\mathbb{Z}[G]$) such that every non-
zero element of $\mathbb{C}[G]$ (and hence $\mathbb{Z}[G]$) is invertible in $\mathcal{U}(G)$.

To summarise, if a groups $G$ satisfies the assumptions of either Theorem A or Theorem C, then the group ring $\mathbb{Z}[G]$ can be embedded in a ring $D$ such that every non-zero element of $\mathbb{Z}[G]$ is invertible in $D$. Our reason for needing such an embedding is contained in Lemma \ref{lemma: ore} below, which will be used in the proof of Theorem A and C in the next section.

\begin{lemma} \label{lemma: ore} Let $G$ be a non-trivial amenable group and assume that the group ring $\mathbb{Z}[G]$ can be embedded in a ring $D$ such that every non-zero element of $\mathbb{Z}[G]$ is invertible in $D$. 
Then for any index set $J$, the sum $\bigoplus_{j \in J}D$ is an injective $\mathbb{Z}[G]$-module  and for all $n\geq 0$
\[    \mathrm{H}^n(G,\bigoplus_{j \in J}D)=0 .   \]

\end{lemma}
\begin{proof} Since every non-zero element of $\mathbb{Z}[G]$ is invertible in $D$ and $G$ is amenable, it follows that $\mathbb{Z}[G]$ is an Ore domain and that $\bigoplus_{i \in I}D$ is a torsion-free divisible $\mathbb{Z}[G]$-module. By \cite[Th. 1]
{VanDeWater}, this implies that $\bigoplus_{j \in J}D$ is an injective $\mathbb{Z}[G]$-module. From this we immediately conclude that $\mathrm{H}^n(G,\bigoplus_{j \in J}D)=0$ for all $n>0$. Since $g-1$ is invertible in $D$ for every $g 
\in G$ we have $\mathrm{H}^0(G,D)=D^G=0$ and hence also $\mathrm{H}^0(G,\bigoplus_{j \in J}D)=0$.
\end{proof}
\begin{remark} \rm Let $G$ be a group that contains a copy of the free group on $2$ generators and assume that $\mathbb{Z}[G]$ embeds into a ring $D$ such that every non-zero element of $\mathbb{Z}[G]$ is invertible in $D$ (for 
example, the group ring of any torsion-free one-relator group embeds into a skew-field by \cite{LewinLewin}). Then $D$ is a torsion-free divisible $\mathbb{Z}[G]$-module. In particular, $(g-1)m \neq 0$ for any non-trivial $g \in G$ 
and non-zero $m \in D$. But as observed in \cite[Ch. 4. Prop. 2.2.]{RomanThesis} this implies that $D$ is not an injective $\mathbb{Z}[G]$-module. Hence the conclusion of the lemma above does not hold for groups that contain non-abelian free subgroups. This also shows that the injective hull $I(D)$ of $D$ viewed as a $\mathbb{Z}[G]$-module cannot be given the structure of a $D$-module extending the ring structure of $D$. Indeed, assume this was the case and take a 
non-zero $m \in I(D)$. Since $I(D)$ is an essential extension of $D$ as an $\mathbb{Z}[G]$-module, there exists a non-zero $r \in \mathbb{Z}[G]$ such that $r\cdot m \in D$. But since $r$ is invertible in $D$ and the ring structure of $D
$ can be extended to an $D$-module structure on $I(D)$, we have $m= (rr^{-1})\cdot m = r^{-1}(r\cdot m) \in D$. Hence $D=I(D)$ and $D$ is an injective $\mathbb{Z}[G]$-module, which is a contradiction. 

\end{remark}

\section{Proofs of the main results}

We are now ready to prove our main results.\\

\noindent \emph{proof of Theorem A and C.} Let $G$ be a group that satisfies either the assumptions of Theorem A or Theorem C. Our task is to prove that $G$ of type $FP$. By Lemma \ref{lemma: direct limit} it suffices to show that 

\[      \varinjlim{\mathrm{H}^n(G,F_i)}=0     \] 
for any vanishing directed system of free $\mathbb{Z}[G]$-modules $\{F_i\}_{i \in I}$. Let $\{F_i\}_{i \in I}$ be such a system. As we discussed in the previous section, we can embed the ring $\mathbb{Z}[G]$ into a ring $D$ such that 
every non-zero element of $\mathbb{Z}[G]$ is invertible in $D$. Consider the short exact sequence

\[     0 \rightarrow \mathbb{Z}[G] \rightarrow D \rightarrow D / \mathbb{Z}[G] \rightarrow 0.    \]
Tensoring with $-\otimes_{\mathbb{Z}[G]}F_i$ for every $i \in I$, we obtain a short exact sequence of vanishing directed systems of left $\mathbb{Z}[G]$-modules
\[         0 \rightarrow \{  F_i\}_{i \in I} \rightarrow  \{  D \otimes_{\mathbb{Z}[G]}F_i\}_{i \in I} \rightarrow \{  D \otimes_{\mathbb{Z}[G]}F_i/F_i\}_{i \in I}   \rightarrow 0.   \]
By considering the associated directed system of long exact cohomology sequences and taking the direct limit, we obtain the exact sequence
\begin{equation*}   \varinjlim{ \mathrm{H}^{n-1}(G, D \otimes_{\mathbb{Z}[G]}F_i/F_i)} \rightarrow  \varinjlim{ \mathrm{H}^{n}(G, F_i)}  \rightarrow \varinjlim{ \mathrm{H}^{n}(G, D \otimes_{\mathbb{Z}[G]}F_i)} .    \end{equation*}
Since $G$ is of type $FP_{n-1}$, it follows from  \cite[Th. 1.3.]{BieriBook}that  $$\varinjlim{ \mathrm{H}^{n-1}(G, D \otimes_{\mathbb{Z}[G]}F_i/F_i)}=0.$$ 
For each $i \in I$, we have $D \otimes_{\mathbb{Z}[G]}F_i \cong \bigoplus_{j \in J} D$ for some index set $J$ depending on $i$. 

Now let $X$ be an acyclic $G$-CW complex with amenable stabilisers and without a free orbit of $n$-cells. If $G$ is amenable we can choose  $X=\{\ast\}$. Associated to the acyclic $G$-CW-complex $X$, there is a convergent stabiliser spectral sequence (e.g.~see \cite[Ch.VIII 
Section 7]{Brown})

\[       E_{1}^{p,q}= \prod_{\sigma \in \Delta_p}\mathrm{H}^q(K_{\sigma}, \bigoplus_{j \in J} D)) \Longrightarrow \mathrm{H}^{p+q}(G, \bigoplus_{j \in J} D).    \] 
Here, $\Delta_p$ is a set of representatives of $G$-orbits of $p$-cells of $X$ and $K_{\sigma}$ is the stabiliser of $\sigma$. Since every $K_{\sigma}$ is amenable, we conclude from Lemma \ref{lemma: ore} that $$\mathrm{H}
^q(K_{\sigma}, \bigoplus_{j \in J} D)=0$$
for every $q>0$, every $p\geq 0$ and every $\sigma \in \Delta_p$.  Hence, $E^1_{p,q}=0$ for all $q>0$.

Since $K_{\sigma}$ is a non-trivial amenable group for every $\sigma \in \Delta_n$ it also follows from Lemma \ref{lemma: ore} that  

\[        E^1_{n,0}=  \prod_{\sigma \in \Delta_n} \mathrm{H}^0(K_{\sigma}, \bigoplus_{j \in J} D) =0.\]
We conclude that the entire $(p,q)$-line with $p+q=n$ of the spectral sequence above is zero. It follows that $\mathrm{H}^{n}(G, \bigoplus_{j \in J} D)=0$ and hence
\[        \varinjlim{ \mathrm{H}^{n}(G, D \otimes_{\mathbb{Z}[G]}F_i)}=0. \]
This implies that $\varinjlim{ \mathrm{H}^{n}(G, F_i)} =0$, 
which proves Theorems A en C.

\qed\\
\noindent \emph{proof of Corollary 2.} Let $\Gamma$ be the fundamental group of a graph of groups with cyclic edge groups and vertex groups that are either subgroups of the additive rational numbers or solvable Baumslag-Solitar 
groups. By Bass-Serre theory (\cite{SerreTrees}), $\Gamma$ acts on a tree $T$ (a $1$-dimensional acyclic $\Gamma$-CW-complex) with cyclic edge (1-cell) stabilisers and vertex (0-cell) stabilisers that are either subgroups of the 
additive rational numbers or solvable Baumslag-Solitar groups. It follows that the stabilisers of $0$-cells have cohomological dimension at most two (e.g.~see \cite{gilden}) while the stabilisers of $1$-cells have cohomological 
dimension at most one. An application of the stabiliser spectral sequence associated to the action of $\Gamma$ on $T$ shows that $\Gamma$, and hence also every subgroup of $\Gamma$, has cohomological dimension at most 
$2$. Since finitely generated groups are of type $FP_1$, the corollary follows from Corollary 1. \qed \\

\noindent \emph{proof of Corollary 3.} Let $H$ be a non-trivial amenable subgroup of $G$ such that $G$ commensurates $H$. Now consider the discrete $G$-set $\mathcal{S}=\coprod_{g\in G} G/H^{g}$ and let $X$ denote the infinite join of $\mathcal{S}$ with itself. By construction, there is a cellular and admissible action of $G$ on $X$ such that the cell-stabilisers are finite intersections of conjugates of $H$. Since $G$ commensurates $H$, we conclude that every cell in $X$ has  a non-trivial amenable stabiliser.  Moreover, since taking joins increases the connectivity, $X$ is acyclic. We conclude that $X$ is an acyclic $G$-CW complex with amenable stabilisers and  no free orbits of $n$-cells, for 
any $n$.  The corollary now follows from Theorem C. \qed \\

We now turn to the proof of Theorem B. \\

\noindent \emph{proof of Theorem B.} First consider the case where $G$ is finitely generated.  In this case it follows from Theorem A that $G$ is of type $FP$. Since amenable groups satisfy the Bass conjecture by \cite{BCM}, we 
conclude from \cite[Section 4.1]{Eckmann2} that the ordinary Euler characteristic  of $G$ coincides with the $L^2$-Euler characteristic of $G$ and is therefore equal to zero (see also \cite{LuckL2Book}). This implies that $\mathrm{H}_1(G,\mathbb{C})\cong 
\mathbb{C}\otimes_{\mathbb{Z}}G/[G,G] \neq 0$, proving that $G$ admits a surjection $\pi$ onto $\mathbb{Z}$. Since $G$ is of type $FP_2$, i.e. almost finitely presented, it follows from \cite[Theorem A]{BieriStrebel} that $G$ can be 
realised as an HNN-extension of a finitely generated base group $K$ with stable letter $t$ mapped to $1 \in \mathbb{Z}$ by $\pi$. Since $G$ is amenable, it does not contain non-abelian free subgroups. This forces the HNN-extension 
to be ascending, by Britton's Lemma. Also, by Theorem A, $K$ is of type $FP$. We claim that $\mathrm{cd}(K)=1$. To see this, assume by contradiction $\mathrm{cd}(K)=2$. Since $K$ must be one ended (since it does not contain 
non-abelian free subgroups) this implies that $\mathrm{H}^1(K,\mathbb{Z}[G])=0$. But by \cite[Th. 0.1]{BrownGeoghegan}, the map $\alpha$ in the Mayer-Vietoris sequence
\[     \mathrm{H}^1(K,\mathbb{Z}[G]) \rightarrow  \mathrm{H}^2(G,\mathbb{Z}[G]) \rightarrow  \mathrm{H}^2(K,\mathbb{Z}[G])   \xrightarrow{\alpha } \mathrm{H}^2(K,\mathbb{Z}[G])      \]
associated to the ascending HNN-extensions $G=K{\ast}_t$ is injective. We therefore conclude that $\mathrm{H}^2(G,\mathbb{Z}[G])=0$, contradicting the assumption that $\mathrm{cd}(G)=2$. This proves our claim, implying that $K
$ is free by Stalling's Theorem \cite{Stallings}. Since $K$ is amenable, this forces $K=\mathbb{Z}=\langle a \rangle$, proving that $G$ equals a solvable Baumslag-Solitar group 
\[  BS(1,m) = \langle a ,t \ | \ t^{-1}at = a^m \rangle      \]
for some non-zero $m \in \mathbb{Z}$.

Now consider the general case. Since every finitely generated group of cohomological dimension $1$ is free by Stallings' theorem it follows from the above that every finitely generated subgroup of $G$ either infinite cyclic or a solvable 
Baumslag-Solitar group. In particular, $G$ is locally solvable and elementary amenable of finite Hirsch length. Since torsion-free elementary amenable groups of finite Hirsch length are virtually solvable by \cite{HillmannLinnel}, it 
follows that $G$ must be solvable. By \cite[Theorem 5]{gilden}, every solvable group of cohomological dimension two is isomorphic to either a solvable Baumslag-Solitar group or a non-cyclic subgroup of the additive rationals.

\qed

\end{document}